\documentclass[12pt,reqno]{amsart}

\usepackage{amsmath}
\usepackage{amssymb}
\usepackage{amsthm}
\usepackage[left=2.7cm,top=2.7cm,right=2.7cm,bottom=2.7cm]{geometry}
\usepackage{enumerate}
\usepackage[usenames]{color}
\usepackage{graphicx}
\usepackage{bbm}

\newtheorem{theorem}{Theorem}[section]

\newcommand{\brac}[1]{\left(#1\right)}

\newcommand{\beq}[1]{\begin{equation}\label{#1}}
\newcommand{\eeq}{\end{equation}}

\newcommand{\set}[1]{\left\{#1\right\}}

\def\E{\mbox{{\bf E}}}

\def\Pr{\mbox{{\bf Pr}}}

\newcommand{\ignore}[1]{}

\def\l{\lambda}
\def\r{\rho}
\def\z{\zeta}
\def\d{\delta}

  \def\d{\delta} 
\def\e{\varepsilon}    
  
\def\z{\zeta}     \def\l{\lambda}
 \def\m{\mu}  
\def\r{\rho}

\def\E{\mbox{{\bf E}}}

\def\Pr{\mbox{{\bf Pr}}}

\begin{document}

\title{A Random Variant of the Game of Plates and Olives}

\author{Andrzej Dudek}
\address{Department of Mathematics, Western Michigan University, Kalamazoo, MI}
\email{\tt andrzej.dudek@wmich.edu}
\thanks{The first author was supported in part by a grant from the Simons Foundation (522400, AD)}

\author{Sean English}
\address{Department of Mathematics, Western Michigan University, Kalamazoo, MI}
\email{\tt sean.j.english@wmich.edu}

\author{Alan Frieze}
\address{Department of Mathematical Sciences, Carnegie Mellon University, Pittsburgh, PA}
\email{\tt alan@random.math.cmu.edu}
\thanks{The third author was supported in part by NSF grant DMS1661063}

\begin{abstract}
The game of plates and olives was originally formulated by Nicolaescu and encodes the evolution of the topology of the sublevel sets of Morse functions.
We consider a random variant of this game. The process starts with an empty table. There are four different types of moves: (1) add a new plate to the table, (2) combine two plates and their olives onto one plate, removing the second plate from the table, (3) add an olive to a plate, and (4) remove an olive from a plate. We show that with high probability the number of olives is linear as the total number of moves goes to infinity. Furthermore, we prove that the number of olives is concentrated around its expectation.
\end{abstract}

\maketitle

\section{Introduction}

The game of plates and olives is a purely combinatorial process that has an interesting application to topology and Morse theory. Morse theory involves the study of topological manifolds by considering the smooth functions on the manifolds. An \emph{excellent Morse function} on the $2$-sphere is a smooth function from $S^2\to \mathbb{R}$ such that all the critical points are non-degenerate (i.e. the matrix of second partial derivatives is non-singular) and take distinct values. 

If $f$ is an excellent Morse function on the sphere, $S^2$, with critical points $x_1,\dots,x_m$ with $f(x_1)<\dots<f(x_m)$, a \emph{slicing} of $f$ is an increasing sequence $a_0,\dots,a_m$ such that $a_0<f(x_1)<a_1<f(x_2)<\dots<a_{m-1}<f(x_m)<a_m$. Then two excellent Morse functions $f$ and $g$, with the same number of critical points, are said to be \emph{topologically equivalent} if for any slicing $a_0,\dots,a_m$ of $f$ and $b_0,\dots,b_m$ of $g$, there is an order-preserving diffeomorphism (i.e. an isomorphism of smooth manifolds) between the sublevel sets $\{x\in S^2\mid f(x)\leq a_i\}$ and $\{x\in S^2\mid g(x)\leq b_i\}$ for each $0\leq i\leq m$. 

Loosely speaking, two excellent Morse functions are topologically equivalent if when their critical values are ordered as mentioned above, both functions have the same types of critical points (in terms of being local minima, maxima, or saddle points), appearing in the same order, and in a rough sense the same location relative to other critical values, which is necessary for the sublevel sets described above to diffeomorphic.

Morse functions on the sphere have exactly $2n+2$ critical points, $n$ of which are saddle points. It was shown in \cite{N} that the sublevel sets $\{f(x)\leq a\}$ are topologically equivalent to either all of $S^2$, or a finite (possibly empty) disjoint union of disks, each with at most a finite number of punctures (i.e. isolated ``missing'' points). As the value $a$ crosses a critical point, one of the following four things will take place: 
\begin{enumerate}[\qquad (1)]
\item a new disk may appear, 
\item two such disks may merge (preserving the punctures in both disks), 
\item a new puncture may appear, or 
\item a puncture may disappear. 
\end{enumerate}
Given a slicing $a_0,\dots,a_m$ of the excellent Morse function $f$, we have that the first sublevel set, $\{f(x)\leq a_0\}=\emptyset$, and the second, $\{f(x)\leq a_1\}$ is a disk. The second to last sublevel, $\{f(x)\leq a_{m-1}\}$ is also a disk, and the last sublevel set, $\{f(x)\leq a_m\}$, is the entire sphere, and this is the only sublevel set that is topologically equivalent to the sphere.

The game of plates and olives was originally formulated by Nicolaescu in \cite{N} and encodes the evolution of the topology of the sublevel sets in a purely combinatorial process in which plates play the role of disks and olives represent punctures in the disks. The moves in the game of plates and olives are designed to resemble exactly the possible transformations that happen when $a$ crosses a critical point: 
\begin{enumerate}[\qquad (1)]
\item add a plate, 
\item combine two plates while keeping all the olives, 
\item add an olive to a plate, or
\item remove an olive from a plate. 
\end{enumerate} 
The game of plates and olives begins at an empty table and ends the first time we return to an empty table, signifying that the level sets of a Morse function start with the empty set and end with the entire sphere.

Let $T^2_n$ denote the number of excellent Morse functions on the $2$-sphere with $n$ saddle points, up to topological equivalence. A lower bound for $T^2_n$ was given by Nicolaescu in \cite{N} by studying walks on Young's lattice. An upper bound on $T^2_n$ was given by Carroll and Galvin in \cite{CG} from studying the game of plates and olives directly. The bounds of these two papers give
\[
(2/e)^{n+o(n)}n^n\leq T^2_n\leq (4/e)^{n+o(n)}n^n.
\]

Here we will study a random variant of the game of plates and olives. 

\section*{The Model}

The process starts with an empty table. There are four different types of moves that can happen in the process. The four main moves are as follows:

\begin{enumerate}
\item[($P^+$)] Add a plate; from every configuration we can add one empty plate to the table.
\item[($P^-$)] Combine two plates; assuming there are at least two plates, we can choose two of them (order does not matter), combine their olives onto one plate, and remove the other plate from the table.
\item[($O^+$)] Add an olive to a plate; we can choose any plate and add one olive to it.
\item[($O^-$)] Remove an olive from a plate; we can choose any non-empty plate and remove one olive from it.
\end{enumerate}

In our model, the plates are distinguishable, but the olives are not. At each time step in the process, one of the available moves will be chosen to be performed uniformly at random.

In addition to the random aspect, our model differs from the game of plates and olives only in that in our model, the plates are distinguishable, and we do not allow for the process to return to an empty table. 

Our main result shows that the number of olives grows linearly with the number of steps in the process, and that the number of olives is concentrated. When we refer to an event occurring with high probability (w.h.p. for short), we mean that the probability of that event goes to $1$ as the total number of moves, $t$, goes to infinity.

\begin{theorem}\label{theorem main thorem}
Let $O_t$ be the total number of olives on the table in the preceding model at time $t$. 
\begin{enumerate}[(a)]
\item There exist absolute constants $C>0$ and $\frac{1}{342}\leq c_1\leq c_2\leq \frac23$ such that
\begin{equation}\label{thm:eq1}	
\Pr(c_1t\leq O_t \leq c_2t) \ge 1 - e^{-Ct}.
\end{equation}
\item Furthermore, there exists an absolute constants $A>0$ such that for every $\d\ge 0$ we have
\begin{equation}\label{thm:eq2}	
\Pr(|O_t-\E(O_t)|\geq \d t )\leq e^{-A\d^2t}
\end{equation}
\item Also, w.h.p. no plate, except for the first plate, has more than $B\log t$ olives at any time, for some absolute constant $B>0$.
\end{enumerate}
\end{theorem}

We prove in Section~\ref{section expected number of olives} that
\begin{equation}\label{eq:eot}
1/342\leq \E(O_t)\leq 2/3.
\end{equation}
Next in Section~\ref{section concentration} we derive the concentration result~\eqref{thm:eq2}, which altogether will imply~\eqref{thm:eq1}. In Section~\ref{section markov chain} we consider an auxiliary Markov chain process. It follows from our proofs that constants $1/342$ and $2/3$ are not optimal. As a matter of fact a computer simulation suggests that the number of olives $O_t$ is concentrated around $c t$, where $c\approx 0.096$.

\section{Bounds on the expected number of olives}\label{section expected number of olives}

\subsection{Lower bound}

We would like to show that the number of olives at a given time grows linearly with time~$t$. Towards this, we will establish two facts:
\begin{itemize}
	\item we expect to return to a single plate a linear number of times, and
	\item each time we return to a single plate, we expect to gain a positive number of olives.
\end{itemize}
This will give us a linear expectation. 

Now let us show that we expect to return to a single plate a linear number of times. If we have $\ell\ge 1$ plates, then the probability we do a plate move is at least
\[
\frac{\binom{\ell}2+1}{2\ell+\binom{\ell}2+1}\geq 1/3.
\]
Let $t_{plate}$ be the random variable that counts the number of plate moves we have after $t$ moves overall. Then
\begin{equation}\label{equation expected number of plate moves}
\E(t_{plate})\geq\sum_{i=1}^t 1/3 = t/3.
\end{equation}
Now let us consider only plate moves to get a lower bound on the random variable $X$, which counts the number of times we transition from two plates to one plate.

We consider a related Markov chain. In this process, we will consider a random walk on the positive integers. We will start this walk at $1$ (plate). If we are currently at $1$, then we will move to $2$ with probability $1$. If we are currently at $2$, we will move to $1$ with probability $1/2$ and to $3$ with probability $1/2$. If we are at $k\geq 3$, we will move to $k-1$ with probability $3/4$ and to $k+1$ with probability $1/4$. This Markov chain will be indexed by time $t_{plate}$ as it only models moves made when there is at least one plate.

Observe that in our model, for $\ell\geq 3$,
\begin{equation}\label{eq:rm_plate}
\Pr(P^-|\text{there are }\ell\text{ plates currently and we perform a plate move})=\frac{\binom{\ell}{2}}{\binom{\ell}2+1}\geq 3/4,
\end{equation}
and
\[
\Pr(P^+|\text{there are }\ell\text{ plates currently and we perform a plate move})=\frac{1}{\binom{\ell}2+1}\leq 1/4.
\]
Thus the Markov chain gives an underestimate for how often we transition from two plates to one plate. Let $N_{1,1}(t_{plate})$ be the random variable that tracks the total number of times this Markov chain returns to a state with a single plate, given that we start at a state with a single plate and a total of $t_{plate}$ plate moves have been made.

By Theorem \ref{theorem Markov chain}, we have that $\E(N_{1,1}(t_{plate}))\geq \E(t_{plate})/19$. Note that $\E(N_{1,1}(t_{plate}))\leq \E(X)$, so
\[
\E(X)\geq \E(t_{plate})/19\geq t/57.
\]

Now we explore what happens each time we transition from two plates to one plate. Consider a state in the process that currently has two plates. If the second plate currently has olives on it, then the probability that next time we make a plate move, plate $2$ still has olives is at least $1/2$. (We can immediately make the plate move.) If the second plate currently has no olives on it, then the probability there is at least one olive on it when we make the next plate move is at least $1/6$. ($1/3$ probability to add an olive, $1/2$ probability to perform a plate move once we have added the olive.) Thus we have at least a $1/6$ chance of adding an olive to the first plate each time we reduce the number of plates to $1$. Let $Y$ be a random variable that counts the number of times we add at least one olive to the first plate from a plate move, given that we transition from $2$ plates to $1$ plate $X$ times. Then
\[
\E(Y)\geq \E(X)/6.
\]

Now we put everything together. We will only consider olives on the first plate. Let $O_t^{(1)}$ denote the number of olives on plate $1$ at time $t$. Let $O^{(1)+}_t$ and $O^{(1)-}_t$ denote the total number of olives that were added to (respectively subtracted from) plate $1$ from $O^+$ (respectively $O^-$) moves. Note that $\E(O^{(1)+}_t-O^{(1)-}_t)\geq 0$ since the probability of performing an $O^+$ move is always at least the probability of performing an $O^-$ move. Finally, let $O^{plate}_t$ denote the total number of olives added to the first plate from plate moves. Then $O^{plate}_t\geq Y$. This gives us that
\beq{Ot}
\E(O_t)\geq \E(O_t^{(1)})=\E(O^{(1)+}_t-O^{(1)-}_t+O^{plate}_t)\geq \E(Y)\geq \E(X)/6\geq t/342.
\eeq

\subsection{Upper bound}
Now we bound the expected value of~$O_t$ from above. Let $O_t^+$ and $O_t^-$ are the random variables that count the number of $O^+$ moves and the number of $O^-$ moves after $t$ total moves, respectively. Clearly,
\[
O_t=O_t^+-O_t^-=t-t_{plate}-2O_t^- \le t-t_{plate}.
\]
Thus, by \eqref{equation expected number of plate moves} we conclude that
\[
\E(O_t)\leq t-\E(t_{plate})\leq 2t/3.
\]
And this proves~\eqref{eq:eot}.

\section{Concentration}\label{section concentration}

Suppose that we transition from a state with two plates to a state with a unique plate at times $t_1,t_2,\ldots,t_m$ and recall that $O_t$ denotes the number of olives at time $t$. Define $t_0:=1$. Let $X_i=O_{t_{i+1}}-O_{t_i}$. Then the $X_i$ are independent random variables and based on the previous section we have $\E(X_i)\geq 1/342$. Then $S_m:=O_{t_m}=\sum_{i=0}^mX_i$. We can argue for concentration of $S_m$ as follows.

Note that from \eqref{equation expected number of plate moves} $\E(t_{plate})\geq t/3$, and by the Chernoff bound we have for any $0\le \delta\le 1$,
\[
\Pr(t_{plate}<(1-\delta)t/3)\leq e^{\frac{-\delta^2t}6}.
\]
As we are not trying to optimize the constant $A$, we can be imprecise here and choose $\delta=1/4$, giving
\begin{equation}\label{equation plate probability}
\Pr(t_{plate}<t/4)\leq e^{-\frac{t}{96}}.
\end{equation}
If $t_{plate}\geq t/4$, then the probability that we start at a unique plate, add plates, and then return to a unique plate before $t$ moves is at least $F_{1,1}(t/4)$, which is the probability that our related Markov chain, defined in Section \ref{section expected number of olives} and studied in Section~\ref{section markov chain}, returns to $1$ at least once in the first $t/4$ moves, assuming it started at $1$. Clearly,
\[
\Pr(X_i\geq k)\leq \Pr(t_{i+1}-t_i\geq k)
\] 
and also
\[
\Pr(t_{i+1}-t_i<k)\geq F_{1,1}(k/4)-\Pr(\text{less than }k/4\text{ plate moves happen after $k$ moves}).
\]
Thus, if $k'=k/4$, then by~\eqref{tail} we get that
\begin{align}
\Pr(X_i\geq k) \leq 1-F_{1,1}(k')+e^{-\frac{k}{96}}&=\sum_{j=\lfloor\frac{k'}2\rfloor}^\infty  4\left( \frac{3}{16}\right)^{j-1} \binom{2j-3}{j-1}+e^{-\frac{k}{96}}\notag \\
&\leq 4\sum_{j=\lfloor\frac{k'}2\rfloor}^\infty \left( \frac{3}{4}\right)^{j-1}+e^{-\frac{k}{96}}\leq C'\z^{k'}+e^{-\frac{k}{96}}\leq C\r^k\label{tail1},
\end{align}
where $\z=3^{1/2}/2<1$, $\r=\max\{\z^{1/4},e^{-1/96}\}$ and $C,C'>0$ are constants. Let $\m_i=\E(X_i)$ and $\m=\m_1+\cdots+\m_m$. Note that we have
\begin{equation}\label{eq:mui}
\m_i\leq \sum_{k=1}^\infty kC\r^k = \frac{C\r}{(\r-1)^2}<\infty.
\end{equation}
Now we can easily prove a concentration result for this situation. We modify an argument from~\cite{FP}. We prove
\beq{conc1}
\Pr(|S_m-\m|\geq \d m)\leq e^{-A\d^2m}
\eeq
for some constant $A>0$. That means we have replaced \eqref{Ot} by a concentration inequality.

We write, for $\l>0$ such that $e^\l<1/\r$,
\[
\E(X_i^2e^{\l X_i})=\sum_{k=0}^\infty k^2e^{\l k}\Pr(X_i=k)\leq C\sum_{k=0}^\infty k^2(\r e^\l)^k\leq \frac{3C}{(1-\r e^\l)^3}.
\]
Now $e^x\leq 1+x+x^2e^x$ for $x\geq 0$, and so, using the above, we have
\[
\E(e^{\l X_i})\leq 1+\l\m_i+\l^2\brac{\frac{3C}{(1-\r e^{\l})^3}}< 1+\l\m_i+\l^2\brac{1+\frac{3C}{(1-\r e^{\l})^3}}.
\]
Since $\Pr(S_m\geq \m_i m+\d m) = \Pr(e^{\l S_m} \geq e^{\l (\m_i m+\d m)})$ and $X_i$s are independent, the Markov bound implies that
\begin{align*}
\Pr(S_m\geq \m_i m+\d m)&\leq e^{-\l(\m_i m+\d m)}\prod_{i=1}^m\E(e^{\l X_i})\\ 
&\leq \exp\set{-\l(\m_i m+\d m)}\cdot \brac{1+\l\m_i+\l^2\brac{1+\frac{3C}{(1-\r e^{\l})^3}}}^m\\
&\leq \exp\set{-\l(\m_i m+\d m)}\cdot \exp\set{\brac{\l\m_i+\l^2\brac{1+\frac{3C}{(1-\r e^{\l})^3}}}m}\\
&= \exp\set{-\l\d m + \l^2\brac{1+\frac{3C}{(1-\r e^{\l})^3}}m}\\
&\leq \exp\set{-\l\d m + \l^2(1+3C\e^{-3})m},
\end{align*}
where $\e=\e(\d)>0$ is a constant such that
\beq{star}
e^\l\leq (1-\e)/\r.
\eeq
Now choose $\l=\d/(2(1+3C\e^{-3}))$ and $\e$ such that \eqref{star} holds. Such a choice of $\e$ is always possible since as $\e\to 0$, $\exp\set{\d/(2(1+3C\e^{-3}))}\to 1$ and $(1-\e)/\r\to 1/\r>1$. Then
\[
\Pr\brac{S_m\geq \m+\d m} \leq
\max_{1\le i\le m}\Pr\brac{S_m\geq \m_i m+\d m)}\leq \exp\set{-\frac{\d^2m}{4(1+3\e^{-3})}}.
\]
To bound $\Pr(S_m\leq \m-\d m)$, we proceed similarly. We have $e^{-x}\leq 1-x+x^2e^x$, so
\[
\E(e^{-\l X_i})\leq 1-\l\m_i+\l^2\brac{\frac{3C}{(1-\r e^{\l})^3}}< 1-\l\m_i+\l^2\brac{1+\frac{3C}{(1-\r e^{\l})^3}}
\]
and
\begin{align*}
\Pr(S_m\leq \m_i m-\d m)&\leq e^{\l(\m_i m-\d m)} \prod_{i=1}^m\E(e^{-\l X_i})\\ 
&\leq \exp\set{\l(\m_i m-\d m)}\cdot \brac{1-\l\m_i+\l^2\brac{1+\frac{3C}{(1-\r e^{\l})^3}}}^m\\
&\leq \exp\set{\l(\m_i m-\d m)}\cdot \exp\set{\left(-\l\m_i+\l^2\brac{1+\frac{3C}{(1-\r e^{\l})^3}}\right)m}\\ 
&\leq \exp\set{-\l\d m + \l^2(1+3C\e^{-3})m},
\end{align*}
and we can proceed as before. This completes the proof of~\eqref{conc1}.

For \eqref{conc1} to be useful, we need to show that w.h.p. $m$ is linear in $t$. We condition on performing a plate move. Let the random variables $\tau_i$ for $1\leq i<\infty$ count how many times we have exactly $i$ plates after $t$ steps, where we only count when the number of plates change. So, if we are at e.g. two plates and we make three olive moves before the next plate move, we only count this as having two plates once. Then $t_{plate}=\sum_{i=1}^\infty \tau_i$ and $\tau_1=m$. Note that we can express $\tau_1$ as a sum of $\tau_2$ indicator random variables that denotes if on the $j$th time we are at two plates, we then transition to one plate. Note that the probability of such a transition is 1/2 (when we have exactly two plates, there is one way to remove a plate and one way to add a plate, giving equal probability of moving to $1$ plate vs. $3$ plates) and so $\E(\tau_1) = \tau_2/2$.

We will consider two cases based on the value of $\tau_2$. If $\tau_2\geq 3t_{plate}/19$, we have from equation \eqref{equation plate probability} and the Chernoff bound,
\begin{align*}
\Pr(\tau_1\leq t/76=(1/3)&(1/4)(3/19)t)\leq \Pr(\tau_1<\tau_2/3)+\Pr(t_{plate}<t/4)\\
&\leq e^{-\tau_2/18}+e^{-t/96}\leq e^{-t_{plate}/114}+e^{-t/96}\leq e^{-t/456}+e^{-t/96}.
\end{align*}

Now, if $\tau_2<3t_{plate}/19$ and $\tau_1<t_{plate}/19$, then $\tau_{\geq 3}:=\sum_{i=3}^\infty \tau_i\geq 15t_{plate}/19$. Now, let $L_{\geq 3}$ be the random variable that counts how many times we have at least three plates and we remove a plate. Due to~\eqref{eq:rm_plate} we have at least a $3/4$ probability of removing a plate whenever we make a plate move, so by the Chernoff bound, we have
\[
\Pr(L_{\geq 3}\leq 3\tau_{\geq 3}/5)\leq e^{-3\tau_{\geq 3}/200}\leq e^{-9t_{plate}/950}\leq e^{-9t/3800}+e^{-t/96}.
\]
Thus w.h.p. when we have at least three plates, we remove plates at least $3/5$ of the time and add them at most $2/5$ of the time. This implies that we must transition from two plates to three plates at least $\tau_{\geq 3}/5$ times to make up for the discrepancy. This implies trivially that $\tau_2\geq \tau_{\geq 3}/5\geq 3t_{plate}/19>\tau_2$, a contradiction. Thus, there exists an absolute constant $D>0$ such that 
\begin{equation}\label{no2}
m=\tau_1\geq \frac{t}{76}\text{ with probability at least }1-e^{-Dt}.
\end{equation}

Now observe that~\eqref{tail1} also implies that for $k = \log_\rho (1/Ct^2) = B\log t$ (for some constant $B>0$) we have
\begin{equation}\label{no1}
\Pr(t_{i+1}-t_i\geq k) \le C\r^k = 1/t^2
\end{equation}
and so
\begin{equation}\label{eq:log}
\Pr\left(\bigcup_{1\le i\le m} (t_{i+1}-t_i\geq k)\right) \le t \cdot  \frac{1}{t^2} = o(1).
\end{equation}
Note that between time $t_{i}$ and $t_{i+1}$ the number of olives at any plate different from the first one is at most $t_{i+1}-t_i$ and so~\eqref{eq:log} implies that w.h.p. no plate, except for the first plate has more than $B\log t$ olives at any time. Part (c) of Theorem \ref{theorem main thorem} follows directly from \eqref{eq:log}.

Now let $T = O_t - O_{t_m} = O_t - S_m$. Then, by~\eqref{tail1} we have $\Pr(T\geq k)\leq C\rho^k$ and the triangle inequality implies
\[
|O_t - \E(O_t)| = |T+S_m - \E(O_t)| \le |S_m - \mu| + |T + \mu - \E(O_t)| =  |S_m - \mu| + |T - \E(T)|.
\]
Thus,
\[
\Pr(|O_t - \E(O_t)| \ge \delta t) \le \Pr(|S_m - \mu| \ge \delta t/2)+\Pr(|T-\E(T)|\geq \d t/2).
\]
Furthermore, since $T\ge 0$ and $\E(T)=O(1)$ (cf.~\eqref{eq:mui}) we get that $\Pr(|T-\E(T)|\geq \d t/2) \le \Pr(T\geq \d t/2)$. Hence, 
\begin{align*}
\Pr(|O_t - \E(O_t)| \ge \delta t) &\le \Pr(|S_m - \mu| \ge \delta t/2)+\Pr(T\geq \d t/2)\\
&\le \Pr(|S_m - \mu| \ge \delta m/2)+\Pr(T\geq \d m/2)\\
& \le e^{-A\d^2m/4} +C\r^{\d m/2}.
\end{align*}
This together with \eqref{no2} proves Theorem \ref{theorem main thorem}(b) and this completes the proof of Theorem~\ref{theorem main thorem}.

\section{A related Markov chain}\label{section markov chain}

\begin{theorem}\label{theorem Markov chain}
	Consider a random walk on the positive integers: If we are currently at $1$, then we will move to $2$ with probability $1$. If we are currently at $2$, we will move to $1$ with probability $1/2$ and to $3$ with probability $1/2$. If we are at $k\geq 3$, we will move to $k-1$ with probability $3/4$ and to $k+1$ with probability $1/4$.
	
	Let $N_{1,1}(t)$ be the random variable that counts the number of times we return to state $1$ when we start the walk at state $1$. Then w.h.p. we have that $N_{1,1}(t)\geq t/19$.
\end{theorem}

\begin{proof}
	Notice that we cannot return to $1$ in an odd number of steps. Let $f_{1,1}(2t)$ be the probability that the first time we return to state~1 after $2t$ steps, given that we start at state~1 i.e. the probability that the \emph{first return time} is $2t$. Let $X_j$ be the location at time~$j$. So $X_0 = X_{2t} = 1$, $X_1 = X_{2t-1}=2$ and $X_j\neq 1$ for each $2<j<2t-2$. Furthermore, we need to control the number of steps at which $X_j=2$. Assume that at exactly $(i+1)$ steps we have that $X_j=2$, where $1\le i \le t-1$. That means 
	\[
	X_1 = X_{1+2a_1} = X_{1+2a_1+2a_2} = \dots = X_{1+2a_1+\dots+2a_{i}} = 2,
	\]
	where $1+2a_1+\dots+2a_{i} = 2t-1$ and $a_j\ge 1$. Hence,
	\begin{align*}
	f_{1,1}(2t)&=\Pr(X_{2t}=1,X_1\neq 1,\dots,X_{2t-1}\neq 1\mid X_0=1)\\
	&=\sum_{i=1}^{t-1}\sum_{a_1+\dots+a_{i}=t-1}\prod_{j=1}^{i}C_{a_j-1}\left(\frac{1}{2}\right)^{i+1}\left(\frac{1}{4}\right)^{t-i-2} \left(\frac{3}{4}\right)^{t-1},
	\end{align*}
	where $C_k = \frac{1}{k+1}\binom{2k}{k}$ is the \emph{Catalan number}. Now the Catalan $i$-fold convolution formula (see, e.g., \cite{R}) gives that
	\[
	\sum_{a_1+\dots+a_{i}=t-1}\prod_{j=1}^{i}C_{a_j-1} = \frac{i}{2t-i-2} \binom{2t-i-2}{t-1}.
	\]
	Thus,
	\[
	f_{1,1}(2t) = \sum_{i=1}^{t-1}  \frac{i}{2t-i-2} \binom{2t-i-2}{t-1} \left(\frac{1}{2}\right)^{i+1}\left(\frac{1}{4}\right)^{t-i-2} \left(\frac{3}{4}\right)^{t-1}
	\]
	and equivalently by replacing~$i$ by $k=t-i-1$ we get
	\[
	f_{1,1}(2t) = 4\left( \frac{3}{16}\right)^{t-1} \cdot \sum_{k=0}^{t-2}  \binom{(t-1)+k}{k} \frac{(t-1)-k}{(t-1)+k} 2^{(t-2)-k}.
	\]
	Now we apply the following identity (see, e.g., (1.12) in~\cite{Gouldv3})
	\[
	\sum_{k=0}^{n} \binom{x+k}{k} \frac{x-k}{x+k} 2^{n-k} = \binom{x+n}{n}
	\]
	with $n = t-2$ and $x = t-1$ to conclude that
	\[
	f_{1,1}(2t) = 4\left( \frac{3}{16}\right)^{t-1} \binom{2t-3}{t-1}.
	\]
	
	Consequently, the probability $F_{1,1}(t)$, given $X_0=1$, that we return to $1$ at some point in the first $t$ steps is given by
	\begin{equation}\label{tail}
	F_{1,1}(t)= \sum_{j=1}^{\lfloor\frac{t}2\rfloor} f_{1,1}(2j) = 
	\sum_{j=1}^{\lfloor\frac{t}2\rfloor}  4\left( \frac{3}{16}\right)^{j-1} \binom{2j-3}{j-1}.
	\end{equation}
	By (5.7) in~\cite{G}, we can calculate the mean time $\overline{T}_{1,1}$, to return to state $1$ after starting at state $1$ by
	\[
	\overline{T}_{1,1}=\sum_{t=0}^\infty \Pr(T_{1,1}\geq t)
	=1+\sum_{t=1}^\infty (1-F_{1,1}(t))
	=1+\sum_{t=1}^\infty\sum_{j=\lfloor\frac{t}2\rfloor+1}^\infty 4\left( \frac{3}{16}\right)^{j-1} \binom{2j-3}{j-1}.
	\]
	Now we will evaluate the latter double sum. First observe that 
	\[
	\sum_{t=1}^\infty\sum_{j=\lfloor\frac{t}2\rfloor+1}^\infty a_j = 2\sum_{j=1}^\infty ja_j - \sum_{j=1}^\infty a_j
	= 2\sum_{j=1}^\infty (j-1)a_j + \sum_{j=1}^\infty a_j
	\]
	assuming that all $a_j$s are nonnegative. Thus,
	\begin{align*}
	\overline{T}_{1,1} &= 1+8\sum_{j=1}^\infty (j-1) \left( \frac{3}{16}\right)^{j-1} \binom{2j-3}{j-1} + 4\sum_{j=1}^\infty \left( \frac{3}{16}\right)^{j-1} \binom{2j-3}{j-1}\\
	&= 1+8\sum_{k=1}^\infty k \left( \frac{3}{16}\right)^{k} \binom{2k-1}{k} + 4\sum_{k=0}^\infty \left( \frac{3}{16}\right)^{k} \binom{2k-1}{k}.
	\end{align*}
	Now we will use the following identities that can be obtained by an application of Newton's binomial theorem (see, e.g., (1.30) and (1.3) in~\cite{Gouldv2}):
	\begin{equation}\label{eq:id1}
	\sum_{k=1}^\infty k\binom{2k}{k} \left( \frac{x}{4}\right)^k = \frac{x}{2(1-x)^{3/2}}
	\end{equation}
	and
	\begin{equation}\label{eq:id2}
	\sum_{k=0}^\infty \binom{2k}{k} \left( \frac{x}{4}\right)^k = \frac{1}{\sqrt{1-x}}
	\end{equation}
	for $|x|<1$.
	Applying~\eqref{eq:id1} with $x=\frac{3}{4}$ and using $\binom{2k-1}{k} = \frac{1}{2}\binom{2k}{k}$ for $k\ge 1$ yields
	\[
	8\sum_{k=1}^\infty k \left( \frac{3}{16}\right)^{k} \binom{2k-1}{k} = 4\sum_{k=1}^\infty k \left( \frac{3}{16}\right)^{k} \binom{2k}{k}
	= 4\cdot \frac{3/4}{2(1-3/4)^{3/2}} = 12.
	\]
	Similarly, we apply~\eqref{eq:id2} to get
	\[
	4\sum_{k=0}^\infty \left( \frac{3}{16}\right)^{k} \binom{2k-1}{k} = 4 + 2\sum_{k=1}^\infty \left( \frac{3}{16}\right)^{k} \binom{2k}{k}
	= 4 + 2\left(\frac{1}{\sqrt{1-3/4}} - 1\right) = 6.
	\] 
	Consequently, 
	\[
	\overline{T}_{1,1} = 1 + 12 + 6 = 19.
	\]

	Let $N_{1,1}(t)$ be the number of times that, given starting at state $1$, we return to state $1$ in the first $t$ moves. By (5.8) in~\cite{G}, we have that w.h.p.
	\[
	\lim_{t\to\infty}N_{1,1}(t)/t=1/\overline{T}_{1,1} = 1/19.
	\]
\end{proof}


\begin{thebibliography}{99}

\bibitem{CG} T.~Carroll and  D.~Galvin,
{\em  The game of plates and olives}, \texttt{arXiv:1711.10670}.

\bibitem{FP} A.M. Frieze and W. Pegden, {\em Travelling in randomly embedded random graphs}, \texttt{arXiv:1411.6596}.

\bibitem{Gouldv2} H.~W. Gould,
{\em Tables of Combinatorial Identities. Eight tables based on seven unpublished manuscript notebooks (1945-1990) 
of H. W. Gould}, Edited and Compiled by Prof. Jocelyn Quaintance, \texttt{http://www.math.wvu.edu/~gould/Vol.2.PDF}.

\bibitem{Gouldv3} H.~W. Gould,
{\em Tables of Combinatorial Identities. Eight tables based on seven unpublished manuscript notebooks (1945-1990) 
of H. W. Gould}, Edited and Compiled by Prof. Jocelyn Quaintance, \texttt{http://www.math.wvu.edu/~gould/Vol.3.PDF}.

\bibitem{G} R.~Gallager, 
{\em Stochastic processes. Theory for applications.} Cambridge University Press, Cambridge, 2013. xxii+536 pp. 
Available online:\\\texttt{http://www.rle.mit.edu/rgallager/documents/CountMarkov.pdf}

\bibitem{R} A.~Regev,
{\em A proof of Catalan's convolution formula},
Integers \textbf{12} (2012), no.~5, 929--934. 

\bibitem{N} L.~I.~Nicolaescu, 
{\em Counting morse functions on the 2-sphere}, 
Compositio Mathematica \textbf{144} (2008), no.~5, 1081--1106.

\end{thebibliography}
\end{document}